\DeclareSymbolFont{AMSb}{U}{msb}{m}{n} 			
\documentclass[12pt,oneside,noamsfonts,a4paper]{amsart}

\usepackage[utf8]{inputenx}  
\usepackage[T1]{fontenc}     

\usepackage{microtype}       

\usepackage[charter,cal=cmcal]{mathdesign}  
\renewcommand{\in}{\smallin} \renewcommand{\notin}{\notsmallin} \renewcommand{\setminus}{\smallsetminus}  

\usepackage[scaled]{helvet}          

\usepackage{mathtools}       

\usepackage{mathscinet}			

\newtheorem{theorem}{Theorem}
\newtheorem{proposition}[theorem]{Proposition}
\newtheorem{lemma}[theorem]{Lemma}
\newtheorem{claim}{Claim}
\newtheorem*{claim*}{Claim}
\newcommand{\claimdone}{\hfill$\blacksquare$\par}
\newtheorem{fact}[theorem]{Fact}
\newtheorem{corollary}[theorem]{Corollary}

\theoremstyle{definition}
\newtheorem{definition}[theorem]{Definition}

\newtheorem{question}[theorem]{Question}

\newcommand{\pw}[1]{\mathcal{P}\left(#1\right)}  
\newcommand{\fin}{\textup{fin}}          
\newcommand{\gen}{\textup{gen}}          
\newcommand{\iter}{\mathbin{*}}          
\newcommand{\os}{\mleft \{ \,}             
\newcommand{\cs}{\, \mright \}}             
\newcommand{\la}{\left\langle \,}           
\newcommand{\ra}{\,\right\rangle}          
\usepackage{mleftright}                       
\newcommand{\card}[1]{\mleft| #1 \mright|}    
\DeclareMathOperator{\non}{non}

\DeclareMathOperator{\cov}{cov}
\DeclareMathOperator{\cof}{cof}
\DeclareMathOperator{\tr}{tr}

\newcommand{\RO}[1]{\mathop{\mathrm{RO}}\left(#1\right)}  
\newcommand{\coem}{\lessdot}

\newcommand{\conc}{^\smallfrown}           
\renewcommand{\restriction}{\mathbin{\!\upharpoonright}}   

\renewcommand{\mid}{\shortmid}             

\title[Mathias--Prikry and Laver type forcing]{Mathias--Prikry and Laver type forcing; \\ Summable ideals, coideals, and $+$-selective filters}

\author{David Chodounsk\'y}
\address{Institute of Mathematics of the Academy of Sciences of the Czech Republic,
\v{Z}itn\'{a}~25, Praha~1, Czech Republic}
\email{david.chodounsky@matfyz.cz}

\author{Osvaldo Guzm\'{a}n Gonz\'{a}lez}
\address{Instituto de Matem\'{a}ticas, Universidad Nacional Aut\'onoma de M\'exico, Apartado Postal 61-3, Xangari, 58089, Morelia,
Michoac\'{a}n, M\'{e}xico}
\email{oguzman@matmor.unam.mx}

\author{Michael Hru\v{s}\'{a}k}
\address{Instituto de Matem\'aticas, Universidad Nacional Aut\'onoma de M\'exico,
\'Area de la Investigaci\'on Cient\'{\i}fica, Circuito exterior, Ciudad Universitaria, 04510, M\'exico, D.\ F., M\'exico}
\email{michael@matmor.unam.mx}
\urladdr{http://www.matmor.unam.mx/~michael/}


\thanks{
	The work of the first author was supported by the GACR project 15-34700L and RVO:\ 67985840. The second and third authors gratefully acknowledge support from \mbox{PAPIIT} grant IN 108014 and \mbox{CONACyT} grant 177758.
}

\date{\today}

\begin{document}

\begin{abstract}
We study the Mathias--Prikry and the Laver type forcings associated with filters and coideals.
We isolate a crucial combinatorial property of Mathias reals, and prove
that Mathias--Prikry forcings with summable ideals are all mutually bi-embeddable.
We show that Mathias forcing associated with the complement of an analytic ideal always adds a dominating real.
We also characterize filters for which the associated Mathias--Prikry
forcing does not add eventually different reals, and show that they are countably generated provided they are Borel.
We give a characterization of $\omega$-hitting and $\omega$-splitting families which retain their property
in the extension by a Laver type forcing associated with a coideal.
\end{abstract}

\maketitle

\section*{Introduction}\label{sec:intro}

\noindent
The Mathias--Prikry and the Laver type forcings were introduced in~\cite{mathias} and~\cite{groszek} respectively.
Recently, properties of these forcings were characterized in
terms of properties of associated filters, see~\cite{jorg-michael,david-lyubomyr,gru-hru-mar,hrusak-minami}.
We continue this line of research, and investigate forcings associated with coideals.

\section{Preliminaries}\label{sec:prelim}

Our notation and terminology is fairly standard.
We give here an overview of basic notions used in this paper.
We sometimes neglect the formal difference between
integer singletons and integers, if no confusion is likely to occur.
We are mostly concerned with filters and ideals
on $\omega$ and on the set of finite sets of integers
$\fin = {[\omega]}^{<\omega}$.
If a domain of a filter or ideal is not specified or obvious,
it is assumed that the domain is $\omega$.
All filters and ideals are assumed to be proper and
to extend the Fr\'echet filter.

For $a, b \subseteq \omega$ we write
$a\subset^* b$ if $a \setminus b \in \fin$,
$a =^* b$ if $a \subset^* b$ and $b \subset^* a$,
$a < b$ if $n < m$ for each $n \in a$ and $m \in b$, and
$a \sqsubseteq b$ if there is $n \in \omega$ such that $a = b \cap n$.

A \emph{tree} $T$ will usually be an initial subtree of the tree of finite
sequences of integers $\left(\omega^{<\omega} , \subseteq \right)$
with no leaves. The space of maximal branches of $T$ is denoted $[T]$.
For $t \in T$ we denote by $T[t]$ the subtree 
consisting of all nodes of $T$ compatible with $t$.
An element $r \in T$ is called the \emph{stem} of $T$ if $r$ is the maximal node of $T$ such that $T = T[r]$.
For $a \subseteq \omega$ we denote by $T^{[a]}$ the set of
all nodes $t \in T$ such that $\card{t} \in a$ (i.e.\ the nodes from levels in $a$).
A node $t \in T$ is a \emph{branching node} of $T$ if $t$ has at least two immediate successors in $T$.
For $\mathcal X \subset \pw{\omega}$ we call $t$ an \emph{$\mathcal X$-branching} node if
$\os i \in \omega \mid t\conc i \in T\cs \in \mathcal X$.
A tree is an \emph{$\mathcal X$-tree} if every node of $T$ is $\mathcal X$-branching.

For $\mathcal X \subset \pw{\omega}$ and $A \subseteq \omega$ we write
$\mathcal X \restriction A$ for the set $\os X \cap A \mid X \in \mathcal X \cs$.
For a filter $\mathcal F$ we denote by $\mathcal F^*$ the dual ideal, and
by $\mathcal F^+$ the complement of $\mathcal F^*$ (i.e.\ the $\mathcal F$ positive sets).
For an ideal $\mathcal I$ we denote $\mathcal I^*$ the dual filter,
$\mathcal I^+ = \mathcal {(I^*)}^+$. A complement of an ideal is called a \emph{coideal}.
We will generally not distinguish between terminology for properties of a filter and of the dual ideal,
i.e.\ statements ``$\mathcal F$ is $\varphiup$'' and  ``$\mathcal F^*$ is $\varphiup$'' are
often regarded as synonymous.
We will sometimes speak of filters on general countable sets as if they were filters on $\omega$.
In these cases statements about these filters are understood as statements about
filters on $\omega$ isomorphic with them.

We call an ideal $\mathcal I$ \emph{summable} if there is a function $\mu \colon \omega \to \mathbb R$
such that $\mathcal I = \os I \subseteq \omega \mid \sum \os \mu(i) \mid i \in I \cs < \infty \cs$.
We say that $\mathcal I$ is tall if $\mathcal I \cap {[A]}^{\omega} \neq \emptyset$
for each $A \in {[\omega]}^\omega$.
An ideal $\mathcal I$ is below an ideal $\mathcal J$ in the \emph{Rudin--Keisler order},
$\mathcal I \leq_\mathrm{RK} \mathcal J$ if there is a function $f\colon \omega \to \omega$
such that $I \in \mathcal I$ iff $f^{-1}[I] \in \mathcal J$ for each $I \subseteq \omega$.
We say that $\mathcal I$ is \emph{Rudin--Blass} below $\mathcal J$, $\mathcal I \leq_\mathrm{RB} \mathcal J$
if the witnessing function $f$ is finite-to-$1$.
The Rudin--Keisler and Rudin--Blass ordering on filters is defined in the same way as
on ideals. Note that for ideals is
$\mathcal I \leq_\mathrm{RK} \mathcal J$ iff $\mathcal I^* \leq_\mathrm{RK} \mathcal J^*$,
and similarly for~$\leq_\mathrm{RB}$.

For a filter $\mathcal F$ we will consider the filter $\mathcal F^{<\omega}$ generated by
sets ${[F]}^{<\omega}$ for $F \in \mathcal F$.
If $\mathcal F$ is a filter on $\omega$, then $\mathcal F^{<\omega}$ is a filter on $\fin$.
Notice that for $X \subset \fin$ is $X \in {\mathcal{F}^{<\omega}}^+$
iff for each $F \in \mathcal F$ there is $a \in X$ such that $a \subset F$,
and iff for each $F \in \mathcal F$ there are infinitely many $a \in X$ such that $a \subset F$.
The elements of ${\mathcal F^{<\omega}}^+$ are sometimes called the \emph{$\mathcal F$-universal sets}.

A filter $\mathcal F$ is a \emph{P$^+$-filter} if for every sequence
$\os X_{n}\mid n\in\omega \cs  \subseteq\mathcal{F}^{+}$ there is a sequence
$Y = \os y_{n} \in {[X_n]}^{<\omega} \mid n \in \omega \cs$ such that $\bigcup Y \in \mathcal F^+$.

The ideal of all meager sets of reals is denoted by $\mathscr M$.
For $f, g \in \omega^\omega$ write $f =^\infty g$ if $\os n \in \omega \mid f(n) = g(n) \cs$ is infinite.
Recall that $\cov (\mathscr M) = \mathfrak b (\omega^\omega, =^\infty)$ and
$\non (\mathscr M) = \mathfrak d(\omega^\omega, =^\infty)$.
Let $V$ be a model of set theory.
We say that $e \in \omega^\omega$ is an \emph{eventually different} real (over $V$)
if $e \neq^\infty f$ for each $f \in \omega^\omega \cap V$.
We say that $d \in \omega^\omega$ is a \emph{dominating} real (over $V$)
if $f <^* d$ for each $f \in \omega^\omega \cap V$.
Every dominating real is an eventually different real.
For every $f \in \omega^\omega$ the set $\os g \in \omega^\omega \mid g =^\infty f\cs$
is a dense G$_\deltaup$ subset of $\omega^\omega$.
The following proposition is well known, the proof is analogous to 
the proof of~\cite[Lemma~2.4.8]{judah-bartoszynski}. 

\begin{proposition}\label{prop:meager_iff_evdiff}
	Let $V$ be a model of set theory. The set $V \cap \omega^\omega$ is meager in $\omega^\omega$
	if and only if there exists an eventually different real over $V$.
\end{proposition}

The Cohen forcing for adding a subset of a set $X \subseteq \omega$ will be denoted $\mathbb C_X$,
and $\mathbb C$ denotes $\mathbb C_\omega$.
The conditions of $\mathbb C_X$ are finite subsets of $X$ ordered by $\sqsupseteq$.
A~Cohen generic real is the union of a generic filter on $\mathbb C_X$.

Let $\mathcal X$ be a family of subsets of $\omega$, typically a filter or a coideal.
The \emph{Mathias--Prikry forcing} $\mathbb M(\mathcal X)$ associated with $\mathcal X$
consists of conditions of the form $(s, A)$ where $s \in \fin$ and $A \in \mathcal X$.
Although we usually assume $s < A$, we do not require it.
The ordering is given by $(s, A) \leq (t, B)$ if $t \sqsubseteq s$, $A \subseteq B$,
and $s \setminus t \subset B$.
Given a generic filter on $\mathbb M(\mathcal X)$, we call the union of the first coordinates
of conditions in the generic filter the \emph{$\mathbb M(\mathcal X)$ generic real}.
Given $\mathcal X$ and $r \subset \omega$, we denote
$G_r(\mathcal{X}) = \os (s, A) \in \mathbb M(\mathcal X) \mid s \sqsubseteq r, r \subseteq A \cs$.
It is easy to see that $r$ is an $\mathbb M(\mathcal X)$ generic real iff
$G_r(\mathcal{X})$ is a generic filter on $\mathbb M(\mathcal X)$.
Properties of $\mathbb{M}(\mathcal F)$ when $\mathcal F$ is an ultrafilter were studied in~\cite{canjar}
and for $\mathcal F$ a general filter in~\cite{hrusak-minami,david-lyubomyr}.
Since $\mathbb{M}(\mathcal F)$ is $\sigmaup$-centered, it always adds an unbounded real.
On the other hand, it was shown that $\mathbb{M}(\mathcal F)$ can be weakly $\omega^\omega$-bounding and even
almost $\omega^\omega$-bounding.
Filters for which $\mathbb{M}(\mathcal F)$ is weakly $\omega^\omega$-bounding
are called Canjar, and these are exactly those filters for which
$\mathcal{F^{<\omega}}$ is a P$^+$-filter.

The \emph{Laver type forcing} associated with $\mathcal X$ is denoted by $\mathbb L(\mathcal X)$.
Conditions in this forcing is trees $T \subseteq \omega^{<\omega}$ with stem $t$
such that every node $s \in T$, $t \leq s$, is $\mathcal X$-branching.
The ordering of $\mathbb L(\mathcal X)$ is inclusion.
Given a generic filter on $\mathbb L(\mathcal X)$,
the generic real is the union of stems of conditions in the generic filter.
The generic real is a function dominating $\omega^\omega \cap V$, unless $\mathcal X \cap \fin \neq \emptyset$.
Properties of $\mathbb{L}(\mathcal F)$ for $\mathcal F$ filter were studied in~\cite{hrusak-minami,jorg-michael}.

For an ideal $\mathcal I$ on $\omega$, the forcing $\left(\pw{\omega} /\mathcal{I} , \subset \right)$
adds a generic $V$-ultrafilter on $\omega$ containing
$\mathcal I^*$, which will be denoted $\mathcal{G}^{\mathcal I}_{\gen}$.
The superscript will be omitted when $\mathcal I$ is apparent from the context.

A family $\mathcal X$ is \emph{$\omega$-hitting} (also called \emph{$\omega$-tall})
if for each countable sequence $\os A_n \in {[\omega]}^\omega \mid n \in \omega \cs$
exists $X \in \mathcal X$ such that $A_n \cap X$ is infinite for each $n \in \omega$.
A family $\mathcal X$ is \emph{$\omega$-splitting}
if for each countable sequence $\os A_n \in {[\omega]}^\omega \mid n \in \omega \cs$
exists $X \in \mathcal X$ such that both $A_n \cap X$ and $A_n \setminus X $
are infinite for each $n \in \omega$.

To conclude the preliminaries let us recall a useful characterization of F$_{\sigmaup}$ ideals.
A \emph{lower semicontinuous submeasure} is a function
$\varphi \colon \pw{\omega} \to \left[ 0, \infty \right]$ such that
$\varphi \left( \emptyset \right) = 0$;
if $A\subseteq B$, then $\varphi\left( A\right) \leq \varphi \left( B\right)$ (monotonicity);
$\varphi \left( A\cup B\right) \leq \varphi \left( A\right) +\varphi \left( B\right)$ (subadditivity);
and $\varphi \left( A\right) = \sup \os \varphi \left( A\cap n\right) \mid n \in \omega \cs$
for every $A \subseteq \omega$ (lower semicontinuity).

\begin{proposition}[Mazur]\label{prop:mazur}
Let $\mathcal{I}$ be an F$_{\sigmaup}$ ideal on $\omega$.
There is a lower semicontinuous submeasure $\varphi$ 
such that $\varphi\left( \os n \cs \right) = 1$ for
every $n\in\omega$, and $\mathcal{I} =  \operatorname{fin} \left(  \varphi\right)$.
\end{proposition}







\section{Mathias like reals and summable ideals}~\label{sec:Summable}
\nopagebreak

The original motivation for this section comes from a question of Ilijas Farah
about the number of $\mathsf{ZFC}$-provably distinct Boolean algebras of the form $\pw{\omega}/\mathcal I$
where $\mathcal I$ is a `definable' ideal~\cite{farah-question}.
Note that $\mathsf{CH}$ implies that all such Boolean algebras 
are isomorphic for F$_\sigmaup$ ideals $\mathcal I$~\cite{just-krawczyk}.
The interpretation of `definability' interesting in this context might be `F$_{\sigmaup\deltaup}$,' `Borel,' or `analytic.'
The basic question was answered by Oliver~\cite{oliver} 
by showing that there are $2^\omega$ many F$_{\sigmaup\deltaup}$ ideals for which
the Boolean algebras $\pw{\omega}/\mathcal I$ are provably nonisomorphic.
However, these constructions are not interesting from the forcing point of view,
the constructed examples are locally isomorphic to $\pw{\omega}/\fin$.
On the other hand, Stepr\={a}ns~\cite{steprans} showed that there are
continuum many coanalytic ideals whose quotients are pairwise forcing not
equivalent.

We are interested in (anti-)classification results about forcings of this form.
The first result in this direction is due to Farah ad Solecki. They showed that the Boolean algebras
$\pw{\mathbb Q}/\textup{nwd}_{\mathbb Q}$ and $\pw{\mathbb Q}/\textup{null}_{\mathbb Q}$
are nonisomorphic and homogeneous, see~\cite{farah-solecki}.
A systematic study of such forcing notions was done by Hru\v{s}\'ak and Zapletal~\cite{hrusak-zapletal}.
They provided several examples of forcings of this form.
Their results imply that for each tall summable ideal
$\mathcal I$ there is an F$_{\sigmaup\deltaup}$ ideal denoted here $\tr_\mathcal I$
such that $\pw{\omega}/\tr_\mathcal I = \mathbb M (\mathcal I^*) \iter \mathbb Q$
for some $\mathbb Q$, a name for a proper $\omega$-distributive forcing notion.
Therefore showing that the Mathias forcings $\mathbb M (\mathcal I^*)$ are
different for various choices of summable ideals $\mathcal I$ seems to be a viable attempt
to provide a spectrum of different forcings $\pw{\omega}/\tr_\mathcal I$.
However, the results of this section show that this approach is likely to fail,
the Mathias forcings for tall summable ideals all mutually bi-embeddable.

Let us start with a general combinatorial characterization of Mathias generic reals.

\begin{definition}\label{Mathias-like-defin}
	Let $V \subseteq U$ be models of the set theory, $\mathcal F \subset \pw{\omega}$
	be a filter in $V$, and $x \in \pw{\omega} \cap U$.
	We say that $x$ is a \emph{Mathias like real for $\mathcal F$ over $V$} if the following two conditions hold;
	\begin{enumerate}
		\item\label{Mathias-like-small} $x \subset^* F$ for each $F \in \mathcal F \cap V$,
		\item\label{Mathias-like-big} ${[x]}^{<\omega} \cap H \neq \emptyset$ for each $H \in {\mathcal{F}^{<\omega}}^+ \cap V$.
	\end{enumerate}
\end{definition}

Notice that an $\mathbb M(\mathcal F)$ generic real is a Mathias like for $\mathcal F$.
It was implicitly shown in~\cite{hrusak-minami} that Mathias
like reals are already almost Mathias generic
-- it is sufficient to add a Cohen real to get the genericity.
This explains why most results concerning the Mathias forcing rely just on the fact
that the generic reals are Mathias like.
We provide the proof of this fact for reader's convenience.

\begin{proposition}\label{prop:Mat-like_is_Mat}
	Let $V \subseteq U$ be models of the set theory, $\mathcal F \subset \pw{\omega}$
	be a filter in $V$, and $x \in \pw{\omega} \cap U$ be a Mathias like real for $\mathcal F$ over $V$.
	Let $c$ be a $\mathbb C_x$ generic real over $U$.
	Then $c$ is an $\mathbb M(\mathcal F)$ generic real over $V$.
\end{proposition}
\begin{proof}
	We need to prove that $G_c(\mathcal F) \cap \mathscr D \neq \emptyset$
	for each dense subset $\mathscr D \in V$ of $\mathbb M(\mathcal F)$,
	i.e.\ to show that the set of conditions forcing this fact is dense in $\mathbb C_x$.
	Choose any condition $s \in \mathbb C_x$. 
	Denote \[H = \os t \setminus s \mid  \left(\exists F \in \mathcal F \mid (t, F) \in \mathscr D\right),  s \sqsubseteq t \cs.\]
	Note that $H \in {{\mathcal F}^{<\omega}}^+$, otherwise there exists $F \in \mathcal F$ such that ${[F]}^{<\omega} \cap H = \emptyset$,
	and the condition $(s,F)$ has no extension in $\mathscr D$.
	Condition~(\ref{Mathias-like-big}) of Definition~\ref{Mathias-like-defin} now implies that
	there exists $(t, F_t) \in \mathscr D$ such that $s \sqsubseteq t$, and $t \setminus s \subset x$.
	Since $x \subset^* F_t$, there is $k \in x$, $t < k$ such that $x \setminus k \subset F_t$.
	Hence $t \cup \os k \cs \in \mathbb C_x$, $t \cup \os k \cs < s$,
	and $t \cup \os k \cs \Vdash (t, F_t) \in G_c({\mathcal F}) \cap \mathscr D$.
\end{proof}

For a poset $P$ we denote by $\RO{P}$ the unique
(up to isomorphism) complete Boolean algebra in which $P$
densely embeds (while preserving incompatibility), 
and $\RO{P}^+$ denotes the set of non-zero elements of $\RO{P}$.
The relation $\coem$ denotes complete embedding of Boolean algebras.

\begin{corollary}\label{Mathias-embedds}
	Let $\mathbb P$ be a forcing adding a Mathias like real for a filter $\mathcal F$.
	\begin{enumerate}
		\item $\RO{\mathbb M(\mathcal F)} \lessdot \RO{\mathbb P \times \mathbb C}$.
		\item If $\mathbb Q$ is a forcing adding a Cohen real,
		then $\RO{\mathbb M(\mathcal F)} \lessdot \RO{\mathbb P \times \mathbb Q}$.
	\end{enumerate}
\end{corollary}
\begin{proof}
	Proposition~\ref{prop:Mat-like_is_Mat} implies that 
	every generic extension via $\mathbb P \iter \dot{\mathbb C}$ 
	contains a generic filter on $\mathbb M (\mathcal F)$ over $V$. 
	Hence there is
	$a \in \RO{\mathbb M (\mathcal F)}^+$ such that
	\begin{equation}\label{eq:emb}\tag*{$\circledast$}
	\RO{\mathbb M (\mathcal F)} \restriction a \coem
	\RO{\mathbb P \iter \dot{\mathbb C}} = \RO{\mathbb P \times \mathbb C},
	\end{equation}
	see e.g.~\cite{vopenka}.
	For each $p \in \mathbb M(\mathcal F)$, the poset
	$\mathbb M(\mathcal F)\restriction p$ is isomorphic to
	$\mathbb M(\mathcal F \restriction F)$ for some $F \in \mathcal F$.
	If $x$ is a Mathias like real for $\mathcal F$, then it is also Mathias like for
	$\mathcal F \restriction F$ for each $F \in \mathcal F$,
	and we can deduce from Proposition~\ref{prop:Mat-like_is_Mat} that the set
	of elements of $\RO{\mathbb M (\mathcal F)}^+$ satisfying~\ref{eq:emb} is dense.
	Since $\mathbb M (\mathcal F)$ is c.c.c.\ we can find $A$, a countable maximal antichain
	of such elements.
	Now
	\begin{multline*}
	\RO{\mathbb M (\mathcal F)} \simeq \prod_{a \in A} \RO{\mathbb M (\mathcal F)}\restriction a
	\lessdot \prod_\omega \RO{\mathbb P \times \mathbb C} \simeq \\
	\simeq \RO{\mathbb P \times \sum_\omega \mathbb C}
	\simeq \RO{\mathbb P \times \mathbb C}.
	\end{multline*}
	To justify the second last isomorphism, we construct a dense embedding $e$ 
	of the poset $\mathbb P \times \sum_\omega \mathbb C$ 
	into the complete Boolean algebra $\prod_\omega \RO{\mathbb P \times \mathbb C}$: 
	If $t$ is an element of the $n$-th copy of $\mathbb C$ in $\sum_\omega \mathbb C$, 
	define $e(p,t)(i) = (p,t)$ if $i = n$, and $e(p,t)(i) = \mathbf{0}$ otherwise.

	If $\mathbb Q$ adds a Cohen generic real, then there exists some
	$a \in \RO{\mathbb C}^+$ such that $\RO{\mathbb C} \restriction a \lessdot \RO{\mathbb Q}$.
	Since $\RO{\mathbb C} \restriction a$ is isomorphic to $\RO{\mathbb C}$,
	the second statement follows from the first one.
\end{proof}

The next lemma states that Mathias like reals behave well
with respect to the Rudin--Keisler ordering on filters.

\begin{lemma}\label{RK-Ml_tranfer}
	Let $\mathcal E, \mathcal F$ be filters on $\omega$,
	let $f\colon \omega \to \omega$ be a function witnessing $\mathcal F \leq_\mathrm{RK} \mathcal E$,
	and $x$ be a Mathias like real for $\mathcal E$. Then $f[x]$ is a Mathias like real for $\mathcal F$.
\end{lemma}

\begin{proof}
It is obvious that $f[x] \subset^* F$ for each $F \in \mathcal F$,
so we need to check only condition~(\ref{Mathias-like-big}) of Definition~\ref{Mathias-like-defin}.
Define $f^*\colon \fin \to \fin$ by
\[f^*(h) = \os a \in \fin \mid f[a] = h \cs.\]
\begin{claim*}
	If $H \in {{\mathcal F}^{<\omega}}^+$, then $\bigcup f^*[H] \in {{\mathcal E}^{<\omega}}^+$.
\end{claim*}
For $E \in \mathcal E$ is $f[E] \in \mathcal F$, and there is $h \in H$ such that $h \subset f[E]$.
Thus $f^*(h) \cap {[E]}^{<\omega} \neq \emptyset$.
\claimdone\medskip
Choose any $H \in {{\mathcal F}^{<\omega}}^+$. Since $x$ is Mathias like for $\mathcal E$,
there exists $a \in \bigcup f^*[H]$ such that $a \subset x$.
Now $f[a] \subset f[x]$ and $f[a] \in H$.
\end{proof}

We focus now on summable ideals.
The following simple observation appears in~\cite{farah-quotients}.

\begin{lemma}\label{Farah-claim}
	Let $\mathcal I$, $\mathcal J$ be tall summable ideals. There exists $A \in \pw{\omega} \setminus \mathcal J^*$
	such that $\mathcal I \leq_\mathrm{RB} \mathcal J \restriction A$.
\end{lemma}

We are now equipped to prove the bi-embeddability result.

\begin{theorem}
	Let $\mathcal I$, $\mathcal J$ be tall summable ideals.
	Then $\RO{\mathbb M(\mathcal I)}$ is completely embedded in $\RO{\mathbb M(\mathcal J)}$.
\end{theorem}

\begin{proof}
	Find $A$ as in Lemma~\ref{Farah-claim} and consider the decomposition
	\[\mathbb M(\mathcal J^*) = \mathbb M(\mathcal J^*\restriction A) \times \mathbb M(\mathcal J^*\restriction (\omega \setminus A)).\]
	The forcing $\mathbb M(\mathcal J^*\restriction A)$ adds a Mathias real for $\mathcal \mathcal J^*\restriction A$.
	Lemma~\ref{RK-Ml_tranfer} implies that it also adds a Mathias like real for $\mathcal I^*$.
	Since $\mathcal J^*\restriction (\omega \setminus A)$ is not an ultrafilter,
	the forcing $\mathbb M(\mathcal J^*\restriction (\omega \setminus A))$ adds a Cohen real.
	The conclusion now follows from Corollary~\ref{Mathias-embedds}.
\end{proof}

This shows that the original plan of creating many essentially different forcings
by using different summable ideals is likely to fail.
However, we still do not know whether the Mathias forcing is the same for every tall summable ideal.

\begin{question}
	Are $\mathbb M(\mathcal J^*)$ and $\mathbb M(\mathcal I^*)$ equivalent forcing
	notions if $\mathcal I$ and $\mathcal J$ are tall summable ideals?
\end{question}

To conclude this section let us mention a related result of Farah~\cite[Proposition 3.7.1]{farah-quotients}.

\begin{proposition}
	Assume $\mathsf{OCA + MA}$. If $\mathcal I$ is a summable ideal, then
	$\pw{\omega}/\mathcal I$ is weakly homogeneous iff $\mathcal I = \fin$.
\end{proposition}

\section{Mathias forcing with coideals}\label{sec:Mathias-coid}


This section deals with the forcing $\mathbb M(\mathcal F^+)$
for $\mathcal F$ a filter on $\omega$.
We are mainly interested in the following question.

\begin{question}
When does $\mathbb{M}\left(\mathcal{F}^{+}\right)$ add dominating reals?
\end{question}

The following fact is well known.

\begin{fact}
Let $\mathcal I$ be an ideal on $\omega$. Then
$\mathbb{M}\left(  \mathcal{I}^{+}\right)  =\pw{\omega} /\mathcal{I}\iter\mathbb{M}( \dot{\mathcal{G}^\mathcal{I}_\gen})$.
\end{fact}

\begin{proposition}\label{adds-dom}
If $\mathcal{I}$ is a Borel ideal and $\pw{\omega} /\mathcal{I}$ does
not add reals, then $\mathbb{M}\left( \mathcal{I}^{+}\right)$ adds a
dominating real.
\end{proposition}

\begin{proof}
First assume that $\mathcal{I}$ is an F$_{\sigmaup}$ ideal.
Let $\varphi$ be a submeasure as in Proposition~\ref{prop:mazur}.

Let $r$ be a $\mathbb{M}\left(\mathcal{I}^{+}\right)$ generic real and
notice that $r \notin \operatorname{fin}(\varphi)$.
In $V[r]$ define an increasing function $g\colon\omega
\to\omega$ by letting \[g(n) = \min\os k \in \omega \mid 2^n \leq\varphi (r \cap k) \cs. \]
We will show $g$ is a dominating real.
Let $\left( s,A\right)  \in\mathbb{M}\left(  \mathcal{I}^{+}\right)$ be a condition and
$f\colon\omega\to\omega$ a function in $V$.
We will extend $\left(s,A\right)$ to a condition that forces that $g$ dominates $f$.
Pick $m\in \omega$ such that $\varphi(s) < 2^m$ and for every $i > m$
choose $t_{i}\subseteq A \setminus f(i)$ such that
$\max\left(t_{i}\right) < \min\left(t_{i+1}\right)$
and $2^i \leq \varphi(t_i) < 2^{i+1}$.
This is possible since $\varphi(A) = \infty$ and the $\varphi$-mass of singletons is $1$.
Put $B=\bigcup_{m < i} t_{i}$, thus $\varphi(B) = \infty$
and $(s,B) \in \mathbb{M}\left(\mathcal{I}^{+}\right)$.
Moreover $\left(s, B\right) \leq\left( s,A\right)$,
and since $\left(s, B\right) \Vdash \dot r \subset s \cup B$
we have that $\left(s, B\right) \Vdash f(i) < g(i)$ for $i >m$.

For the general case let $\mathcal{I}$ be an analytic  ideal such
that $\pw{\omega}/\mathcal{I}$ does not add reals.
If $\mathcal{G}^{\mathcal{I}}_{\gen}$ is not a P-point, then it is not a Canjar filter (see e.g.\,\cite{canjar}),
and $\mathbb{M}\left(\mathcal{I}^{+}\right) = \pw{\omega} /\mathcal{I}\iter\mathbb{M}\left(\mathcal{G}^{\mathcal{I}}_{\gen}\right)$
will add a dominating real.
In case $\mathcal{G}^\mathcal{I}_{\gen}$ is a P-point, then by
\cite[Theorem 2.5]{hrusak-verner}  $\mathcal{I}$ is locally F$_{\sigmaup}$
and $\mathbb{M}\left( \mathcal{I}^{+}\right)$
adds a dominating real as demonstrated in the first part of the proof.
\end{proof}

\begin{question}
	Is there a Borel ideal $\mathcal I$ such that $\mathbb{M}\left( \mathcal{I}^{+}\right)$ 
	does not add a~dominating real? 
\end{question}

It is easy to see that in every generic extension by $\mathbb{M}(\mathcal F^+)$
the ground model set of reals is meager, and thus $\mathbb{M}(\mathcal F^+)$ always adds
an eventually different real.
In~\cite{hrusak-verner} Michael Hru\v{s}\'{a}k and
Jonathan Verner asked the following question.

\begin{question}
Is there a Borel ideal $\mathcal{I}$ on $\omega$ such that
$\pw{\omega}/\mathcal{I}$ adds a Canjar ultrafilter?
\end{question}

We answer this question in negative.

\begin{lemma}
If $\mathcal{I}$ is an ideal on $\omega$ such that $\mathcal{G}^\mathcal{I}_{\gen}$ is a P-point, then
$\pw{\omega}  /\mathcal{I}$ does not add reals.
\end{lemma}

\begin{proof}
Let $A\in\mathcal{I}^{+}$ and $r$ a name such that $A\Vdash \dot{r}\in
\omega^{\omega}$. Let $\mathcal{G}_{\gen}$ be a $\pw{\omega}/\mathcal{I}$
generic filter such that $A\in\mathcal{G}_{\gen}$
and for every $n\in\omega$ we can find $A_{n}\in\mathcal{G}_{\gen}$ such that
$A_{n}\leq A$ and $A_{n}$ decides $\dot{r}\left(n\right)$. Since
$\mathcal{G}_{\gen}$ is a $P$-point, there is $B\in\mathcal{G}_{\gen}$ such that
$B\subseteq^{\ast}A_{n}$ for every $n\in\omega$ (note that we can assume $B$
is a ground model set since $\mathcal{G}_{\gen}$ is generated by ground model
sets). Clearly $B\leq A$ and forces $\dot{r}$ to be a ground model real.
\end{proof}

\begin{corollary}
If $\mathcal{I}$ is an analytic ideal
then $\mathcal{G}^\mathcal{I}_{\gen}$ is not a Canjar ultrafilter.
\end{corollary}

\begin{proof}
By the previous proposition if $\pw{\omega}/\mathcal{I}$
adds new reals then the generic filter is not a Canjar
ultrafilter.
Assume no new reals are added. By Proposition~\ref{adds-dom},
$\mathbb{M}\left(\mathcal{I}^{+}\right)$ adds a dominating real and
$\mathcal{G}_{\gen}$ is not Canjar.
\end{proof}

\section{Mathias--Prikry forcing and eventually different reals}~\label{sec:Mathis-EDreals}

We turn our attention towards the forcing $\mathbb M(\mathcal F)$
for a filter $\mathcal F$.
Our goal is the characterization of filters for which this forcing does not add eventually different reals.

A filter $\mathcal F$ is \emph{$+$-Ramsey}~\cite{laflamme} if for each $\mathcal F^+$-tree $T$ there is a branch
$b \in [T]$ such that $b[\omega] \in \mathcal F^+$.

\begin{definition}
Let $\mathcal F$ be a filter on $\omega$.
We say that $\mathcal{F}$ is \emph{$+$-selective}
if for every sequence $\os X_{n}\mid n\in\omega \cs  \subseteq\mathcal{F}^{+}$ there is a selector
\[S=\os  a_{n} \in X_n \mid n\in\omega\cs  \in\mathcal{F}^{+}.\]
\end{definition}

Every $+$-Ramsey filter is $+$-selective and every $+$-selective filter is a P$^+$-filter.

Let $M$ be an extension of the universe of sets $V$.
We say that $r \in \omega^\omega \cap M$ is an
eventually different real over $V$
if the set $\os n \in \omega \mid r(n) = f(n) \cs$ is finite
for each $f \in \omega^\omega \cap V$.
We say that a forcing $\mathbb P$ does not add
an eventually different real iff there is no eventually
different real over $V$ in any generic extension by forcing $\mathbb P$.

\begin{theorem}\label{thm:ED-Mathias}
Let $\mathcal F$ be a filter. The following are equivalent;
\begin{enumerate}
	\item\label{item:Mat-no-ed}
		Forcing $\mathbb{M}\left(  \mathcal{F}\right)$ does not add an eventually different real,
	\item\label{item:+selective}
		$\mathcal{F}^{<\omega}$ is $+$-selective,
	\item\label{item:+Ramsey}
		$\mathcal{F}^{<\omega}$ is $+$-Ramsey.
\end{enumerate}
\end{theorem}

\begin{proof}
The implication
 (\ref{item:+Ramsey}) $\Rightarrow$ (\ref{item:+selective}) is clear.
We start with (\ref{item:+selective}) $\Rightarrow$ (\ref{item:Mat-no-ed}).

Let ${\mathcal{F}^{<\omega}}$ be $+$-selective and $x$ be
an $\mathbb{M}(\mathcal{F})$ name for a function in  $\omega^\omega$.
Enumerate $\fin = \la s_i \mid i \in \omega \ra$ such that $\max s_i \leq i$ for each $i \in \omega$.
Let $\os a_i \mid i \in \omega \cs$ be a partition of $\omega$ into infinite sets,
and denote by $a_i(k)$ the $k$-th element of $a_i$.
For $k \in \omega$ let
\begin{multline*}
X_k = \bigl \{ \, t \in \fin \mid k < \min t \text{ and }
\forall i < k \colon \exists h^t_i(k) \in \omega \colon \exists F \in \mathcal F \colon  \\
(s_i \cup t, F) \Vdash \dot x(a_i(k)) = h^t_i(k) \, \bigr \}.
\end{multline*}
\begin{claim*}
$X_k \in {\mathcal F^{<\omega}}^+$ for each $k \in \omega$.
\end{claim*}
Let $k \in \omega$.
We need to show that for each $G \in \mathcal F$ there exists $t \in X_k$ such that $t \subset G$.
Put $t_0 = \emptyset$, $F_0 = G \setminus (k+1)$, and for $i < k$ proceed with an inductive construction as follows.

Suppose $t_i$, $F_i$ were defined, we will define $t_{i+1}$, $F_{i+1}$, $h^t_i(k)$.
Find a condition $p = (s_i \cup t_{i+1},F_{i+1} ) < ( s_i \cup t_i, F_i )$
and $h^t_i(k) \in \omega$
such that $p \Vdash \dot x(a_i(k)) = h^t_i(k)$.
Finally put $t = t_k$, and notice that $t \in X_k$, $t \subset G$.
\claimdone\medskip

Let $S \in {\mathcal F^{<\omega}}^+$ be a selector for $\la X_k \mid k \in \omega \ra$
guaranteed by the $+$se\-le\-cti\-vity of ${\mathcal F^{<\omega}}$.
Define  $g ;\: \omega \to \omega$ by $g(a_i(k)) = h^t_i(k)$ if $t \in S$ and $i < k$.
We claim that $\Vdash \card{ \os  g(n) = \dot x(n) \mid n \in \omega\cs } = \omega$.

Let $(s_i,G)$ be a condition and $n$ be an integer.
There exists $k > n, i$ and $t \in X_k \cap S$ such that $t\subset G$.
Thus there is $F \in \mathcal F$ such that
\[(s_i \cup t, F) \Vdash \dot x(a_i(k)) = h^t_i(k) = g(a_i(k)).\]
Put $p = (s_i \cup t, F \cap G) < (s_i, G)$.
Now $n < k \leq a_i(k)$ and $p \Vdash \dot x(a_i(k)) = g(a_i(k))$.
\smallskip

To prove (\ref{item:Mat-no-ed}) $\Rightarrow$ (\ref{item:+Ramsey})
assume $\mathbb{M}\left(\mathcal{F}\right)$ does not add an eventually
different real.
Let $T$ be an ${\mathcal{F}^{<\omega}}^+$-tree and $r$ be an $\mathbb{M}(\mathcal F)$ generic real.
For $n \in \omega$ let $O_n = \os a \in [T] \mid \exists m > n : a(m) \subset r \setminus n \cs$.
Note that each such $O_n$ is an open dense subset of $[T]$.
Now $G = \bigcap \os O_n \mid n \in \omega \cs$ is a dense G$_\deltaup$ set,
and Proposition~\ref{prop:meager_iff_evdiff} implies that there exists some $b \in G \cap V$.
We claim that $b$ is the desired branch for which $b[\omega] \in {\mathcal{F}^{<\omega}}^+$.
Otherwise there is $F \in \mathcal{F}$ such that $b[\omega] \cap F^{<\omega} = \emptyset$,
which contradicts $r \subset^* F$ and $\card{{[r]}^{<\omega} \cap b[\omega]} = \omega $.
\end{proof}

The last part of the proof in fact demonstrated the following.

\begin{theorem}
	Let $V \subseteq U$ be models of the set theory, $\mathcal F \subset \pw{\omega}$
	be a filter in $V$. If $U$ contains a Mathias like real for $\mathcal F$
	but no eventually different real over $V$, then $\mathcal F$ is $+$-Ramsey.
\end{theorem}

The implication (\ref{item:+selective}) $\Rightarrow$ (\ref{item:+Ramsey}) of Theorem~\ref{thm:ED-Mathias}
can be proved directly with the same proof as is used in~\cite[Lemma 2]{pawlikowski}.
Although this implication holds true for filters of the form $\mathcal{F}^{<\omega}$,
this is not the case for filters in general.
The filter on $2^{<\omega}$ generated by complements of $\subseteq$-chains and $\subseteq$-antichains is an F$_{\sigmaup}$ $+$-selective filter
which is not $+$-Ramsey.

The following proposition is a direct consequence of~\cite[Theorem~2.9]{laflamme}.

\begin{proposition}
Let $\mathcal F$ be a Borel filter. $\mathcal F$ is $+$-Ramsey if and only
if $\mathcal F$ is countably generated.
\end{proposition}

\begin{corollary}\label{cor:ED-or-Cohen}
If $\mathcal F$ is a Borel filter on $\omega$ and $\mathbb M(\mathcal F)$
does not add an eventually different real, then $\mathbb M(\mathcal F)$
is forcing equivalent to the Cohen forcing.
\end{corollary}

\begin{proof}
If $\mathbb M(\mathcal F)$ does not add an
eventually different, then the Borel $\mathcal{F}^{<\omega}$ is
$+$-Ramsey and hence countably generated. Thus $\mathcal F$ is also countably
generated and $\mathbb M(\mathcal F)$ has a countable dense subset.
\end{proof}

It is not hard to see that any forcing of size less than
$\cov(\mathscr{M})  $ can not add an eventually different real, so we have
another proof of the following well known result,

\begin{corollary}
If $\mathcal{I}$ is a Borel ideal which is not countably generated then
$\cov(\mathscr{M}) \leq \cof(\mathcal{I})$.
\end{corollary}

Corollary~\ref{cor:ED-or-Cohen} can be derived directly from~\cite[Conclusion 9.16]{shelah-nep},
which says that if a Suslin c.c.c.\ forcing adds a non-Cohen real,
then it makes the set of ground model reals meager.
See also~\cite[Corollary 3.5.7]{zapletal-idealized}.

\section{Laver type forcing}\label{sec:Laver}

We will address the question of preserving hitting families with Laver
type forcing. Since every forcing adding a real destroys some maximal almost disjoint family,
it only makes sense to ask for survival of hitting families with some additional properties.
Preservation of $\omega$-hitting and $\omega$-splitting families with Laver forcing $\mathbb L$
was studied in~\cite{dow}.
A characterization of the strong preservation of these properties with forcing $\mathbb L(\mathcal F)$
for a filter $\mathcal F$ was given in~\cite{jorg-michael}.
Preservation of spliting familes with $\mathbb L(\mathcal F)$ was also studied in~\cite{dilip-jorg}.
We utilize methods used in~\cite{dow} to characterize $\omega$-hitting and
$\omega$-splitting families for which the Laver forcing $\mathbb L(\mathcal F^+)$
preserves the $\omega$-hitting and the $\omega$-splitting property.

\begin{definition}
	Let $\mathcal X \subset \pw{\omega}$ be a family of sets
	and let $\mathcal F$ be a filter on $\omega$.
	We say that $\mathcal X$ is \emph{$\mathcal F^+$-$\omega$-hitting}
	if for every countable set of functions
	$\os f_n\colon \omega \to \omega \mid n \in \omega\cs$
	such that $f_n[\omega] \in  \mathcal{F}^+$ for each $n \in \omega$,
	there exists $X \in \mathcal X$ such that $f_n[X] \in \mathcal{F}^+$
	for each $n \in \omega$.
\end{definition}

Obviously, every $\mathcal F^+$-$\omega$-hitting family must be $\omega$-hitting.

\begin{proposition}
	Let $\mathcal F$ be a filter on $\omega$ and let $\mathcal X \subset \pw{\omega}$.
	The following are equivalent;
	\begin{enumerate}
		\item\label{hit-prop:condition} $\mathcal X$ is $\mathcal F^+$-$\omega$-hitting,
	 	\item\label{hit-prop:omega} $\mathbb L(\mathcal F^+)$ preserves ``$\check{\mathcal X}$ is $\omega$-hitting.''
	\end{enumerate}
\end{proposition}
\begin{proof}
Start with (\ref{hit-prop:condition}) implies (\ref{hit-prop:omega}).
For conditions $S, T \in \mathbb L(\mathcal F^+)$, where the stem of
$T$ is $r \in \omega^k$, we write $S <^n T$ if
$S < T$ and $S\cap \omega^{k+n} = T \cap \omega^{k+n}$.


Let $\theta$ be a large enough cardinal and let $M \prec H_\theta$
be a countable elementary submodel containing $\mathcal F$.
Let $X \in \mathcal X$ be such that $f[X] \in \mathcal F^+$ for each
$f\colon \omega \to \omega$, $f \in M$ such that $f[\omega] \in \mathcal F^+$.

\begin{claim}\label{claim:laver-induct_step}
Let $A \in M$ be an $\mathbb L(\mathcal F^+)$-name,
and $S \in \mathbb L(\mathcal F^+) \cap M$ be a condition
such that $S \Vdash \dot A \in {[\omega]}^\omega$.
There exists $S' <^0 S$ such that for each $T'< S'$ there is $t \in T'$
such that $S'[t] \in M$ and $S'[t] \Vdash \check X \cap \dot{A} \neq \emptyset$.
\end{claim}

Since $S$ is countable and $A$ is a name for an infinite set,
we can inductively build a sequence $\os \la t_n, k_n, R_n \ra \mid n \in \omega \cs \in M$
such that

\begin{itemize}
	\item $t_n \in S$, $k_n \in \omega$, $R_n \in \mathbb L(\mathcal F^+)$,
	\item $R_n <^0 S[t_n]$,
	\item $R_n \Vdash k_n \in \dot A$,
	\item $k_n \neq k_m$ for $n \neq m$,
	\item $I = \{ t_n \mid n \in \omega \}$ is a maximal antichain in $S$.
\end{itemize}
Put $S' = \bigcup \os R_n \mid k_n \in X \cs$.
Let $r$ be the stem of $S$.
We only need to show that for each $s \in S$ such that $r \leq s <t$
for some $t \in I$, the set
$\os i \in \omega \mid s\conc i \in S'\cs$ is in $\mathcal F^+$.
Define a function $f \colon \omega \to \omega$ in $M$
by $f\colon k_n \mapsto i$ if $t_n \geq s\conc i$ for $n \in \omega$,
and $f \colon k \mapsto 0$ otherwise.
Note that  $\os i \in \omega \mid s\conc i \in S \cs \subseteq f[\omega] \in \mathcal F^+$
since $I$ is maximal bellow $s$.
Thus $f[X] = \os i \in \omega \mid s\conc i \in S'  \cs \in \mathcal F^+$.
\claimdone\medskip

Let $T \in \mathbb{L}(\mathcal F^+) \cap M$ be a condition with stem $r$.
Enumerate $\os A_n \mid n \in \omega \cs$ all $\mathbb L(\mathcal F^+)$-names
belonging to $M$
such that $\Vdash \dot A_n \in {[\omega]}^\omega$ for each $n \in \omega$.
We will inductively construct a fusion sequence of conditions
$\os T_n \mid n \in \omega \cs$
starting with $T_0 = T$ such that
\begin{itemize}
	\item $T_{n+1} <^n T_n$ for each $n \in \omega$,
	\item for each $T' < T_n$ there is $t \in T'$ such that
	      $T_n[t] \in M$ and \\
	      $T_n[t] \Vdash \dot{A}_n \cap \check{X} \neq \emptyset$.
\end{itemize}

Suppose that $T_n$ is constructed and use the inductive hypothesis to find
a maximal antichain
$J \subset \os t \in T_n \mid n + \card{r} < \card{t}, T_n[t] \in M\cs$
 in $T_n$.
For each $t \in J$ use Claim~\ref{claim:laver-induct_step} for $S = T_n[t]$
and $A = A_{n+1}$ to get $T'_n[t] <^0 T_n[t]$ as in
the statement of the claim.
Now $T_{n+1} = \bigcup \os T'_n[t] \mid t \in J \cs$ is as required.

Once this sequence is constructed
put $R = \bigcap \os T_ n \mid n \in \omega \cs \in \mathbb L(\mathcal F^+)$.
Now $R \Vdash \dot{A}_n \cap \check{X} \neq \emptyset$ for each $n \in \omega$,
and the implication is proved.
\medskip

For the other direction, assume there are functions
$\os f_n \colon \omega \to \omega \mid n \in n \cs$
such that $f_n[\omega] \in \mathcal F^+$,
and for each $X \in \mathcal X$ there is $n \in \omega$
such that $f_n[X] \in \mathcal F^*$.
Fix $\os b_n \in {[\omega]}^\omega \mid n \in \omega \cs$, a partition of $\omega$ into infinite sets.
Let $\dot{\ell}$ be a name for the  $\mathbb L(\mathcal F^+)$ generic real, and define a name for $\dot{A}_n^k \subset \omega$
by declaring $\dot{A}_n^k = f_n^{-1}\left[\dot\ell[b_n\setminus k]\right]$ for each $k,n \in \omega$.
Inductively define $T \in \mathbb L(\mathcal F^+)$ such that $t \conc i \in T$ iff $i \in f_n[\omega]$
for $t \in T^{[b_n]}$. Notice that $T$ forces that $\dot{A}_n^k$ is infinite for each $k,n \in \omega$.

Take any $X \in \mathcal X$ and let $S < T$ be a condition with stem $r$.
There is $n \in \omega$ such that $f_n[X] \in \mathcal F^*$.
Put
\[S' = S \setminus \os s \in T \mid \exists t \in T^{[b_n]}, r < t : \exists i \in f[X] : t\conc i \subseteq s  \in S\cs.\]
Note that $S' \in \mathbb L(\mathcal F^+)$ since we removed only
$\mathcal F^*$ many immediate successors of each splitting node of $S$.
Also notice that $S' \Vdash X \cap \dot{A}^{\card{r}}_n = \emptyset$.
Thus for each $X \in \mathcal X$ the condition $T$ forces
that $X$ does not have infinite intersection with all sets $\dot{A}_n^k$,
and $\mathcal X$ is not $\omega$-hitting in the extension.
\end{proof}

We can formulate the ``splitting'' version of the previous result. 
A similar result for $\mathbb L(\mathcal F)$, where $\mathcal F$ is a filter, 
is contained in~\cite[Section~6]{dilip-jorg}.

\begin{definition}
	Let $\mathcal X \subset \pw{\omega}$ be a family of sets
	and let $\mathcal F$ be a filter on $\omega$.
	We say that $\mathcal X$ is \emph{$\mathcal F^+$-$\omega$-splitting}
	if for every countable set of functions
	$\os f_n\colon \omega \to \omega \mid n \in \omega\cs$
	such that $f_n[\omega] \in  \mathcal{F}^+$ for each $n \in \omega$,
	there exists $X \in \mathcal X$ such that $f_n[X], f_n[\omega \setminus X] \in \mathcal{F}^+$
	for each $n \in \omega$.
\end{definition}

Again, every $\mathcal F^+$-$\omega$-splitting family is $\omega$-splitting.
The same proof as before with the obvious adjustments gives us the following.

\begin{proposition}
	Let $\mathcal F$ be a filter on $\omega$ and let $\mathcal X \subset \pw{\omega}$.
	The following are equivalent;
	\begin{enumerate}
		\item\label{split-prop:condition} $\mathcal X$ is $\mathcal F^+$-$\omega$-splitting,
	 	\item\label{split-prop:omega} $\mathbb L(\mathcal F^+)$ preserves ``$\check{\mathcal X}$ is $\omega$-splitting,''
	\end{enumerate}
\end{proposition}

\bibliography{exploited}
\bibliographystyle{amsalpha}

\end{document}